\documentclass[english,11pt]{smfart}

\usepackage{etex}

\usepackage{amsbsy}
\usepackage{amsmath,amsfonts,amssymb,amsthm,mathrsfs,mathtools}

\usepackage{bm}

\usepackage[a4paper,vmargin={3cm,3cm},hmargin={3.5cm,3.5cm}]{geometry}
\usepackage[font=sf, labelfont={sf,bf}, margin=1cm]{caption}
\usepackage{graphicx}
\usepackage{epsfig}%pour les eps
\usepackage{latexsym}%encore des symboles
\usepackage{xcolor}
\usepackage{ae,aecompl}
\usepackage{soul,framed}
\usepackage{comment}

\usepackage{xcolor}
\usepackage[pdfpagemode=UseNone,bookmarksopen=false,colorlinks=true,urlcolor=blue,citecolor=blue,citebordercolor=blue,linkcolor=blue]{hyperref}
\usepackage{smfhyperref}
\usepackage[capitalize]{cleveref}

\usepackage{pstricks}
\usepackage{enumerate}
\usepackage{tikz,animate,media9}						%%%%%%%%%%
\usepackage{todonotes}
\usepackage{pifont}
\usepackage{bm,marvosym}
\usepackage{algorithm}
\usepackage{algorithmic}

\usepackage[most]{tcolorbox}
\usepackage{enumitem}

\usepackage{bbm}

\usepackage{tcolorbox}

\definecolor{aleacolor}{rgb}{0.16,0.59,0.78}

\hypersetup{
	breaklinks,
	colorlinks=true,
	linkcolor=aleacolor,
	urlcolor=aleacolor,
	citecolor=aleacolor}

\newcount\colveccount
\newcommand*\colvec[1]{
	\global\colveccount#1
	\begin{pmatrix}
		\colvecnext
	}
	\def\colvecnext#1{
		#1
		\global\advance\colveccount-1
		\ifnum\colveccount>0
		\\
		\expandafter\colvecnext
		\else
	\end{pmatrix}
	\fi
}

\newtcolorbox[auto counter]{boxedalgorithm}[2]{%
	enhanced,
	breakable,
	colback=black!2,
	colframe=black!65,
	colbacktitle=black!8,
	coltitle=black,
	fonttitle=\bfseries,
	title={Algorithm~\thetcbcounter. #2},
	label={#1},
	attach boxed title to top left={xshift=3mm,yshift=-2mm},
	boxed title style={%
		colback=white,
		colframe=black!65,
		boxrule=0.6pt,
		arc=1.5mm,
		left=1.5mm,
		right=1.5mm,
		top=0.5mm,
		bottom=0.5mm
	},
	boxrule=0.6pt,
	arc=2mm,
	left=3mm,
	right=3mm,
	top=4mm,
	bottom=3mm,
	before skip=1em,
	after skip=1em
}

\newcommand{\ndN}{\mathbb{N}}
\newcommand{\ndZ}{\mathbb{Z}}

\renewcommand{\Pr}[1]{\mathbb{P}(#1)}

\newcommand{\Prb}[1]{\mathbb{P}\left(#1\right)}

\newcommand{\Ex}[1]{\mathbb{E}[#1]}

\newcommand{\Exb}[1]{\mathbb{E}\left[#1\right]}

\newcommand{\one}{{\mathbbm{1}}}

\newcommand{\eqdist}{\,{\buildrel \mathrm{d} \over =}\,}

\newcommand{\Pois}{\mathrm{Poisson}}

\newtheorem{theorem}{Theorem}[section]

\newtheorem{corollary}[theorem]{Corollary}
\newtheorem{proposition}[theorem]{Proposition}
\newtheorem{lemma}[theorem]{Lemma}

\newtheorem{definition}[theorem]{Definition}

\numberwithin{equation}{section}

\title{\textbf{Linear time sampling of P\'olya trees}}
\date{}

\author{Benedikt Stufler}
\address[Benedikt Stufler]{Vienna University of Technology}
\email{benedikt.stufler at tuwien.ac.at}

\begin{document}

\vspace {-0.5cm}

\begin{abstract}
	We present an algorithm for exact-size sampling of uniformly random P\'olya trees with linear expected runtime. 
\end{abstract}

%\tableofcontents

\maketitle

\section{Introduction}

Efficient samplers for discrete structures are a useful tool in probabilistic combinatorics. They make it possible to simulate large random objects, estimate shape statistics, test conjectures, and suggest new phenomena.

A classical algorithm of Devroye~\cite{MR2888318} samples Bienaym\'e--Galton--Watson trees conditioned on their size in expected linear time. Through the enriched-tree framework, this algorithm has become the basis for a broad class of exact-size samplers for random tree-like combinatorial structures~\cite{zbMATH07755473}. Further milestones for efficient samplers of trees and related structures include~\cite{zbMATH06774439,zbMATH05039060,zbMATH06244239,zbMATH06683520,zbMATH06109904}.

These methods, however, do not directly apply to unordered unlabelled trees. The basic example is the class of P\'olya trees, that is, rooted trees considered up to isomorphism. Their recursive description is a multiset construction rather than an ordered sequence or a finite-type decoration of the offspring of each vertex. This feature prevents a direct application of the enriched-tree version of Devroye's algorithm.

There is a well-developed Boltzmann framework for unlabelled structures. Bodirsky, Fusy, Kang and Vigerske~\cite{MR2810913} introduced P\'olya--Boltzmann samplers based on cycle index sums and cycle pointing, covering in particular unordered trees and related unlabelled structures. These samplers have expected running time linear in the size of the Boltzmann object and are well suited for approximate-size generation. For exact-size generation, however, the direct Boltzmann rejection or truncation approach is not linear in the square-root singularity regime of P\'olya trees. Since the counting sequence satisfies $a_n \sim c n^{-3/2}\rho^{-n}$, the critical exact-size rejection strategy has probability of order $n^{-3/2}$ of hitting size $n$, leading to quadratic expected cost.

Recent work of Bartholdi and Diaconis~\cite{zbMATH08207340} introduced a novel Markov chain sampler for P\'olya trees. Its projection to isomorphism classes has the uniform distribution as stationary distribution. Thus, after sufficiently many steps, it provides an approximate uniform sampler. The currently available analysis, however, does not give a rigorous mixing-time bound.

A connection between P\'olya trees and branching processes was established by Panagiotou and Stufler~\cite{MR3773800}. They showed that the global shape of a large uniform P\'olya tree is governed by a large conditioned critical Bienaym\'e--Galton--Watson tree together with additional attached forests. This connects P\'olya trees to  branching processes. Nevertheless, this representation does not directly facilitate the use of Devroye's algorithm.

The present paper closes this gap. We describe a sampler that produces a uniform random P\'olya tree with a given number $n$ of vertices in expected time $O(n)$, provided elementary arithmetic operations and primitive random variate generators are counted as constant-time operations.

Our approach is inspired by ideas of Sportiello~\cite{Sportiello2021}, who obtained linear-time exact samplers for irreducible context-free structures by separating a critical bridge from conditionally independent subcritical pieces and by moving the main rejection step to the level of this bridge. P\'olya trees do not fall under that finite-type context-free setting due to the multiset construction when decomposing the branches at the root.

Let us finally mention two related approximation results. Stufler~\cite{zbMATH08207337} showed that sampling a uniform P\'olya tree with $n$ vertices and forgetting the root yields an approximation, in total variation, to the uniform unrooted unlabelled tree with $n$ vertices. Hence an exact sampler for rooted P\'olya trees is also useful for the asymptotic simulation of unrooted unlabelled trees. During the writing of the present work, Fusy and Pivoteau~\cite{fusy2026leapgeneratorscompositionschemes} obtained leap generators for composition schemes, including P\'olya trees. These generators have linear time complexity and produce exact-size objects whose distribution is asymptotically uniform in total variation, although perfect uniformity is generally lost. A fusion of this approach with Devroye's algorithm is given in parallel work~\cite{fusypanagiotou2026}. This is complementary to the exact-uniform sampler developed in the present work.

The approach of the present manuscript is robust and we will extend it to unlabelled outerplanar graphs and other classes of unlabelled graphs in subsequent work.

\subsection*{Organisation}
Section~\ref{sec:numerics} deals with numerical estimates for relevant parameters. Section~\ref{sec:elemen} describes samplers for elementary distributions appearing in the Boltzmann sampling framework. Section~\ref{sec:boltz} states a Boltzmann sampler and verifies its expected runtime. Section~\ref{sec:tilt} describes helper distributions used to randomise Boltzmann weights. Section~\ref{sec:bridge} describes a sampling procedure for random bridges involving randomised Boltzmann weights. Section~\ref{sec:poly} states the final sampler for P\'olya trees and verifies its expected runtime.

\subsection*{Notation}

We  let $\ndN = \{1, 2, \ldots\}$ and $\ndN_0 = \{0, 1, 2, \ldots\}$ denote the sets of positive and nonnegative integers. For $n \in \ndN_0$ we set $[n] = \{1, \ldots, n\}$. The $n$th coefficient of a power series $f(z)$ is denoted by $[z^n]f(z)$. All considered random variables are defined on a common probability space whose measure is denoted by $\mathbb{P}$. Unspecified limits are taken as $n$ tends to infinity. For a random variable $X$ and an event $E$ we use the notation $\Ex{X, E} = \Ex{X \one_E}$ for the expectation of $X$ multiplied by the indicator random variable $\one_E$ for the event $E$. We use $\eqdist$ to denote equality in distribution.

\subsection*{Computational model}

All complexity bounds are stated in a unit-cost real-arithmetic RAM model. Arithmetic operations, comparisons, and evaluations of $\exp$, $\log$, and $\cosh$ are counted as unit-cost operations. The generation of independent $\mathrm{Uniform}([0,1])$ and $\Pois(1)$ random variables, fixed-parameter geometric random variables, and uniform integers from intervals is also counted as having unit cost.

\section{Numerical approximations}

\label{sec:numerics}

The ordinary generating series
\[
	A(z) = \sum_{k \ge 1} a_k z^k
\]
of P\'olya trees satisfies the recursive equation
\begin{align}
	\label{eq:funk}
	A(z) = z \exp(A(z)) \Phi(z)
\end{align}
for
\[
	 \Phi(z) = \exp \left( \sum_{i \ge 2} A(z^i) / i \right).
\]
See the pioneering work by P\'olya~\cite{MR1577579}. Otter~\cite{MR0025715} derived the asymptotic formula
\begin{align}
	\label{eq:coef}
	a_n \sim c_A n^{-3/2} \rho^{-n}, \qquad n \to \infty
\end{align}
for constants $c_A \approx 0.439924$ and $\rho \approx 0.338321$.  Furthermore, Otter~\cite{MR0025715} verified that
\begin{align}
	\label{eq:arone}
	A(\rho) = 1.
\end{align}

\begin{proposition}[{\cite[Sec.~VII.5, Note~VII.21]{MR2483235}}]
	\label{pro:rho}
	We may approximate $\rho$ up to precision $2^{-k}$ using a number of arithmetic operations that grows polynomially in $k$.
\end{proposition}
Equation~\eqref{eq:arone} implies that for all $k \ge 1$
\begin{align}
	\label{eq:akbound}
	a_k \le \rho^{-k}.
\end{align}
This readily yields explicit error bounds for Taylor approximations of~$A(z)$:
\begin{corollary}
	\label{co:taylor} 
	For all $k \ge 1$ and $0 \le t < \rho$
	\[
		0 \le A(t) - \sum_{i=1}^k a_i t^i \le \frac{(t/\rho)^{k+1}}{1 - t/\rho}.
	\]
\end{corollary}
\begin{proof}
We have by~\eqref{eq:akbound}
\[
	0 \le A(t) - \sum_{i=1}^k a_i t^i = \sum_{i \ge k+1} a_i t^i \le \sum_{i \ge k+1} (t/\rho)^i = \frac{(t/\rho)^{k+1}}{1 - t/\rho}.
\]
\end{proof}
\noindent We set for $k \ge 1$
\[
	\omega_k = \sum_{d \mid k} d a_d.
\]
The following recursive relation of coefficients allows us to efficiently calculate these Taylor approximations.
\begin{proposition}
	\label{pro:coef}
	We have $a_1 =1$ and for $k \ge 2$
	\begin{align*}
		a_k = \frac{1}{k-1}\sum_{j=1}^{k-1} a_j \omega_{k-j}.
	\end{align*}
	In particular, the coefficients $a_1, \ldots, a_k$ may be computed in $O(k^2)$ arithmetic operations.
\end{proposition}
\begin{proof}
	It is clear that $a_1=1$. Differentiating~\eqref{eq:funk} yields
	\begin{align}
		\label{eq:comp}
		zA'(z) = A(z) \left(1 + \sum_{i \ge 1} z^i A'(z^i)\right).
	\end{align}
	It is easy to see that for all $k \ge 1$
	\[
		[z^k]\sum_{i \ge 1} z^i A'(z^i) = \sum_{d \mid k} d a_d = \omega_k.
	\]
	Hence comparing coefficients on $z^k$ in~\eqref{eq:comp} yields
	\begin{align*}
		k a_k = a_k + \sum_{j=1}^{k-1} a_j \omega_{k-j}.
	\end{align*}
	Hence
	\begin{align*}
		a_k = \frac{1}{k-1}\sum_{j=1}^{k-1} a_j \omega_{k - j}.
	\end{align*}
	Thus, treating arithmetic operations as having a unit cost, this allows us to compute $a_1, \ldots, a_k$ at cost $O(k^2)$.
\end{proof}

\noindent The series $\Phi(z)$ has radius of convergence $\sqrt{\rho}$. Note that $\rho < \rho^{3/4} < \sqrt{\rho}$.

\begin{proposition}
	\label{pro:phi}
	For all integers $r\ge1$ and $\ell\ge2$ and all
	$0<s\le\rho^{3/4}$,
	\begin{align}
		\label{eq:phier} 0 &\le \log\Phi(s) - \sum_{i=2}^{\ell} \frac{1}{i} \sum_{j=1}^{r}a_j s^{ij} \\
		&\le \frac{\rho^{(r+1)/2}}{(1-\sqrt{\rho})(1-\rho^{3/2})} + \frac{\rho^{3(\ell+1)/4}}{\rho(1-\sqrt{\rho})(\ell+1)(1-\rho^{3/4})}. \nonumber
	\end{align}
	Moreover, $\Phi(s)$ may be approximated with absolute error at
	most $2^{-k}$ using $O(k^2)$ arithmetic operations, uniformly for
	$0<s\le\rho^{3/4}$.
\end{proposition}

\begin{proof}
	We split the error according to
	\begin{align*}
		\log\Phi(s) - \sum_{i=2}^{\ell}\frac{1}{i} \sum_{j=1}^{r}a_j s^{ij}  = \sum_{i=2}^{\ell}\frac{1}{i} \left(A(s^i)-\sum_{j=1}^{r}a_j s^{ij}\right) + \sum_{i>\ell}\frac{A(s^i)}{i}.
	\end{align*}
	For the first sum, Corollary~\ref{co:taylor} gives
	\begin{align*}
		\sum_{i=2}^{\ell}\frac{1}{i} \left(A(s^i)-\sum_{j=1}^{r}a_j s^{ij}\right) 
		&\le \sum_{i=2}^\ell \frac{1}{i} \frac{ (s^i/\rho)^{r+1} }{1 - s^i / \rho}\\
		&\le \frac{1}{1-\sqrt{\rho}} \sum_{i\ge2}(s^i/\rho)^{r+1} \\
		&= \frac{(s^2 / \rho)^{r+1}}{(1 - \sqrt{\rho})(1 - s^{r+1})} \\
		&\le \frac{\rho^{(r + 1)/2}}{(1 - \sqrt{\rho})(1 - \rho^{3/2})}.
	\end{align*}
	Here we used $r\ge1$, and hence $s^{r + 1} \le s^2 \le \rho^{3/2}$. Moreover, for $0\le t\le\rho^{3/2}$, Equation~\eqref{eq:akbound} implies
	\begin{align}
		\label{eq:abound}
		A(t) \le \sum_{j \ge 1} (t/\rho)^j = \frac{t}{\rho - t} \le \frac{t}{\rho(1 - \sqrt{\rho})}.
	\end{align}
	It follows that
	\begin{align*}
		\sum_{i>\ell}\frac{A(s^i)}{i} &\le \frac{1}{\rho(1 - \sqrt{\rho})} \sum_{i>\ell}\frac{s^i}{i} \\
		&\le \frac{s^{\ell + 1}}{\rho(1 - \sqrt{\rho})(\ell + 1)(1 - s)}  \\
		&\le \frac{\rho^{3(\ell + 1)/4}}{\rho(1 - \sqrt{\rho})(\ell + 1)(1 - \rho^{3/4})}. 
	\end{align*}
	This proves~\eqref{eq:phier}.

	Once the coefficients have been precomputed using Proposition~\ref{pro:coef} in $O(r^2)$ arithmetic operations, the finite double sum
	\[
		L_{r,\ell}(s) = \sum_{i=2}^{\ell} \frac{1}{i} \sum_{j=1}^{r}a_j s^{ij}
	\] 
	may be evaluated in $O(r \ell)$ arithmetic operations.
	
	Let $E_{r,\ell}$ denote the right-hand side of \eqref{eq:phier}, so that \eqref{eq:phier} states
	\[
		L_{r,\ell}(s) \le \log\Phi(s) \le L_{r,\ell}(s) + E_{r,\ell}.
	\]
	By the mean value theorem and monotonicity of the exponential function it follows that 
	\[
		0 \le \Phi(s) - \exp(L_{r,\ell}) \le \exp(L_{r,\ell} + E_{r,\ell})E_{r,\ell}.
	\]
	The upper bound decreases exponentially fast. Hence an error bound by $2^{-k}$ is obtained with $r = O(k)$ and $\ell = O(k)$.
\end{proof}

\section{Elementary samplers}

\label{sec:elemen}

\begin{proposition}
	\label{pro:poiss}
	Fix a constant $x_0>0$. Suppose that, uniformly for $0<x\le x_0$, the value $x$ may be approximated up to an error $2^{-k}$ using a number of arithmetic operations that grows at most polynomially in $k$. Then a $\Pois(x)$ random variable may be generated in uniformly bounded expected time.
\end{proposition}
\begin{proof}
	Choose a fixed integer $M>x_0$. Generate independent random variables
	\[
		N_1, \ldots, N_M \longleftarrow \Pois(1)
	\]
	and set
	\[
		N = N_1 + \ldots + N_M.
	\]
	Then $N \eqdist \Pois(M)$.
	
	Since $0<x/M<1$ and $x/M$ inherits a uniform polynomial-time approximation scheme from $x$, \cite[Appendix B, Lem.~7]{zbMATH06195431} allows us to sample a $\mathrm{Bernoulli}(x/M)$ random variable in uniformly bounded expected time. 
	
	Conditionally on $N$, generate independent random variables
	\[
		V_i \longleftarrow \mathrm{Bernoulli}(x/M), \qquad 1\le i\le N,
	\]
	and return
	\[
	W=\sum_{i=1}^N V_i.
	\]
	This way, $W \eqdist \Pois(x)$.
	
	The integer $M$ is fixed, so generating $N$ has bounded expected cost.
	Moreover, $\mathbb E[N]=M$, and the Bernoulli sampling cost is uniformly
	bounded. Hence the expected runtime is bounded uniformly for $0<x\le x_0$.
\end{proof}

We let $\Pois_{\ge 1}(x)$ sample from a Poisson distribution with parameter $x > 0$ that is conditioned to be positive. If we let $\Pois_{\ge 1}(x)$  repeatedly call $\Pois(x)$ until obtaining a positive value, the number of attempts becomes large for small $x$. Hence we may use a different approach.

\begin{corollary}
	\label{co:pplus}
	If we may  approximate a constant $0<x\le1$ up to an error $2^{-k}$ using a number of arithmetic operations that grows at most polynomially in $k$, then we may generate a $\Pois_{\ge 1} (x)$ random variable in expected constant time. If the polynomial speed of approximation is uniform for a family of values for $x$ all satisfying $0<x\le1$, then the expected runtime remains uniformly bounded for that family.
\end{corollary}
\begin{proof}
Using the $\Pois(x)$ sampler from Proposition~\ref{pro:poiss}, repeat the following procedure until it is accepted:
\[
	Z \longleftarrow 1 + \Pois(x),
\]
and accept $Z$ with probability $1/Z$.

For every $j \ge 1$, the probability that one attempt produces $Z=j$ and is accepted is
\[
	\frac{e^{-x}x^{j-1}}{(j-1)!}\frac{1}{j} =  \frac{e^{-x} x^{j-1}}{j!}.
\]
The probability that an attempt is accepted is therefore
\[
	\sum_{j\ge1}\frac{e^{-x}x^{j-1}}{j!} = \frac{1-e^{-x}}{x}.
\]
Hence, conditional on acceptance, the probability of returning $j$ is
\[
	\frac{e^{-x}x^{j-1}/j!}{(1-e^{-x})/x} = \frac{x^j}{j!(e^{x}-1)},
\]
which is the probability mass function of $\Pois_{\ge1}(x)$. 

Finally, the probability $(1-e^{-x})/x$ that an attempt is accepted satisfies
\[
\frac{1-e^{-x}}{x} \ge 1-e^{-1}
\]
since $x \le 1$ and the function $x \mapsto (1 - e^{-x})/x$ is decreasing. Hence the expected number of attempts is bounded uniformly in $0< x \le 1$.  By Proposition~\ref{pro:poiss}, each attempt has a uniformly bounded expected runtime. Hence our sampler for $\Pois_{\ge 1}(x)$ has uniformly bounded expected runtime.
\end{proof}

For any $0<s<\sqrt{\rho}$ we let $\mathrm{MAXINDEX}(s)$ produce a random non-negative integer $K \ge 0$ with distribution given by
\[
\Prb{K \le k } = \Phi(s)^{-1} \exp\left( \sum_{i = 2}^k A(s^i)/i \right).
\]
\begin{proposition}
	\label{pro:kunif}
	Uniformly for all $0 < s \le \rho^{3/4}$ we may sample $\mathrm{MAXINDEX}(s)$ in expected constant runtime.
\end{proposition}
\begin{proof}
	Set $s_0 = \rho^{3/4}$. With
	\[
		C = \frac{1}{\rho(1-\sqrt{\rho})},
	\]
	we have by~\eqref{eq:abound} that
	\begin{align}
		\label{eq:alinear}
		A(t)\le Ct
	\end{align}
	for all $0\le t\le\rho^{3/2}$.
	
	Fix a number $q$ with $s_0 < q <1$, and an integer $B \ge 1$ sufficiently
	large that
	\begin{align}
		\label{eq:bchoice}
		\frac{Cq^2}{2B(1-q)}<1.
	\end{align}
	These choices do not depend on $s$. For $0<s\le s_0$ and $j\ge2$ set
	\[
		p_{s,j}  = \frac{A(s^j)}{Bj(1-q)q^{j-2}}.
	\]
	Since $s^j\le s^2\le\rho^{3/2}$, it follows from~\eqref{eq:alinear} and $s<q$ that
	\[
		p_{s,j} \le \frac{Cs^j}{Bj(1-q)q^{j-2}} \le \frac{Cq^2}{Bj(1-q)} \le \frac{Cq^2}{2B(1-q)} < 1.
	\]
	Thus $p_{s,j}$ is a valid probability.
	
	Consider the following procedure. First generate
	\[
		N \longleftarrow \Pois(B).
	\]
	Conditionally on $N$, independently for each $1 \le r \le N$, generate an integer $J_r \ge 2$ according to
	\[
		\Prb{J_r = j} = (1-q)q^{j-2}, \qquad j \ge 2,
	\]
	and accept $J_r$ with probability $p_{s,J_r}$. Return the largest
	accepted integer, and return $0$ if no integer was accepted.
	
	For a single trial and every $j \ge 2$, the probability of producing and
	accepting the value $j$ is
	\begin{align}
		\label{eq:acceptedindex}
		(1-q)q^{j-2}p_{s,j} = \frac{A(s^j)}{Bj}.
	\end{align}
	Let $K$ denote the value returned by the procedure. For any $k \in \ndN_0$, conditionally on $N$ we consequently have
	\[
		\Prb{K\le k\mid N} = \left(1-\frac{1}{B} \sum_{j \ge \max(2,k+1)} \frac{A(s^j)}{j}\right)^N.
	\]
	Using the probability generating function of the Poisson distribution yields
	\begin{align*}
		\Prb{K\le k} &= \exp\left( -\sum_{j \ge \max(2,k+1)}\frac{A(s^j)}{j} \right)\\
		&= \exp\left( -\sum_{j\ge2}\frac{A(s^j)}{j} +\sum_{j=2}^k\frac{A(s^j)}{j}\right)\\
		&= \Phi(s)^{-1} \exp\left( \sum_{j=2}^k \frac{A(s^j)}{j} \right),
	\end{align*}
	where the last sum is interpreted as empty when $k \le 1$. Hence the returned value has the law prescribed for $\mathrm{MAXINDEX}(s)$.
	
	It remains to verify the uniform runtime bound. The random variable $N$ has a  Poisson distribution with the fixed parameter $B$, and hence bounded expected generation time and expectation $\Ex{N} = B$. Likewise, the random variables $J_r$ have a geometric distribution with the fixed parameter $q$, and thus bounded expected generation time~\cite{zbMATH06195431}. Moreover,
	\begin{align}
		\label{eq:fin}
		\Ex{\log J_r} < \infty.
	\end{align}
	By Corollary~\ref{co:taylor} we may uniformly for $0 < s \le s_0$ and $j \ge 2$ approximate the probability $p_{s,j} = \frac{q^2}{Bj(1-q)}\left(\frac{s}{q}\right)^j \frac{A(s^j)}{s^j}$ up to an error of at most $2^{-k}$ using $O(k^2 + \log j)$ arithmetic operations.
	
	This allows us to apply~\cite[Appendix B, Lem.~7]{zbMATH06195431}, yielding that conditionally on $J_r=j$ we may generate a $\mathrm{Bernoulli}(p_{s,j})$ random variable in expected $O(1 + \log j)$ time, with an implied constant independent of $s$ and $j$. Averaging over  $J_r$ and using~\eqref{eq:fin} it follows that each of the $N$ trials has a uniformly bounded expected runtime. Since $\Ex{N} = B$, the expected runtime of the entire procedure is bounded uniformly for all $0 < s \le \rho^{3/4}$.
\end{proof}

\section{Boltzmann samplers}

\label{sec:boltz}

The Boltzmann distribution of P\'olya trees with parameter $0<t\le \rho$ produces a P\'olya tree $T$ with an arbitrary number of vertices $|T|$ with probability $t^{|T|}/ A(t)$. 

It was shown by~\cite[Sec. 5]{MR2810913} that the following procedure $\Gamma A(t)$ samples according to this distribution.

\begin{boxedalgorithm}{alg:gammaa}{Boltzmann sampler $\Gamma A(t)$}
			\begin{enumerate}
				\item Start with a single root vertex $v$
				\item $X_t \longleftarrow \Pois(A(t))$
				\item For $i=1\ldots X_t$: 
				\subitem Let $T_i \longleftarrow  \Gamma A(t)$
				\item $F \longleftarrow \Gamma B(t)$
				\item Attach the forest $T_1, \ldots, T_{X_t}$ and the forest $F$ to $v$ by adding an edge from $v$ to the root of each tree
				\item Return the resulting tree with root vertex $v$
			\end{enumerate}
\end{boxedalgorithm}

\noindent Here each call of $\Pois(x)$ produces a fresh independent sample from a Poisson distribution with parameter $x>0$.  The recursive sampler uses the following helper sampler $\Gamma B(s)$ for $0<s < \sqrt{\rho}$ that produces a random forest of P\'olya trees.

\begin{boxedalgorithm}{alg:gammab}{Boltzmann sampler $\Gamma B(s)$}			
			\begin{enumerate}
				\item Start with an empty forest $F$
				\item $I \longleftarrow \mathrm{MAXINDEX}(s)$
				\item For $i=2\ldots I-1$:
				\subitem Let $M_i \longleftarrow \Pois(A(s^i)/i)$ 
				\item If $I \ge 2$ set $M_{I} \longleftarrow \Pois_{\ge 1}(A(s^{I})/I)$
				\item For $i=2 \ldots I$:
				\subitem For $j=1 \ldots M_i$:
				\subsubitem \quad Let $T_{i,j} \longleftarrow \Gamma A(s^i)$
				\subsubitem \quad Add $i$ identical copies of $T_{i,j}$ to $F$
				\item Return $F$
			\end{enumerate}
\end{boxedalgorithm}

Let $0<s<\sqrt{\rho}$ be given. The number of vertices $Y_s$ in the forest produced by $\Gamma B(s)$  is an upper bound for the following quantities:
\begin{itemize}
	\item The number of recursive calls to $\Gamma A$ in the procedure.
	\item The largest exponent $J$ appearing in a term $A(s^J)$ or $s^J$ that the sampler needs to evaluate to compute the argument for a $\Pois$,  $\Pois_{\ge1}$ or $\mathrm{MAXINDEX}$ sampler.
\end{itemize}
To see this, simply note that each call to $\Gamma A$ produces at least one vertex, and we create $i$ identical copies of the return value of $\Gamma A(s^i)$ in any instance of $\Gamma B(s)$. 

We let $F_s$ denote the forest generated by $\Gamma B(s)$.
The probability generating series for the number of vertices $Y_s := |F_s|$ is given by
\[
\Exb{ x^{Y_s} } = \Phi(sx) / \Phi(s).
\]
Since $\Phi(x)$ has radius of convergence $\sqrt{\rho}$ it follows that for $0<s < \sqrt{\rho}$ this distribution has finite exponential moments. Its expectation is given by
\[
	\Exb{Y_s} = s \Phi'(s) / \Phi(s) = \sum_{i \ge 2} A'(s^i)s^{i}.
\]
In particular, for any $0<s_0<\sqrt{\rho}$ we have
\begin{align}
	\label{eq:eys}
	\sup_{0 < s \le s_0} \Exb{Y_s} = \sum_{i \ge 2} A'(s_0^i)s_0^{i} < \infty.
\end{align}

\begin{lemma}
	\label{le:boltzb}
	For $0<s\le\rho^{3/4}$ the expected runtime for a call to $\Gamma B(s)$ is uniformly bounded.
\end{lemma}
\begin{proof}
	By the discussion above, the number of recursive calls to $\Gamma A$, the largest exponent $J$ for which a value $A(s^J)$ has to be evaluated when computing a parameter for an elementary sampler, and the total size of generated forests are bounded by a constant multiple of $1 + Y_s$. Moreover, by~Proposition~\ref{pro:poiss},  Corollary~\ref{co:pplus}, and Proposition~\ref{pro:kunif} each call to the elementary random variate generators has uniformly bounded expected runtime. Since $s \le \rho^{3/4} < \sqrt\rho$, Equation~\eqref{eq:eys} gives
	\[
	\sup_{0 < s \le \rho^{3/4}} \Exb{Y_s} < \infty.
	\]
	Hence the expected runtime of $\Gamma B(s)$ is uniformly bounded.
\end{proof}

\begin{proposition}
	\label{pro:invariant}
	For any integer $y \ge 0$ with $[z^y]\Phi(z) >0$ the conditional distribution $(F_s \mid Y_s = y)$ does not depend on $s$. In other words, for all $0<s,t< \sqrt{\rho}$ we have
	\[
		(F_s \mid Y_s = y) \eqdist (F_t \mid Y_t = y).
	\]
\end{proposition}
\begin{proof}
	The condition $[z^y]\Phi(z) >0$ ensures that both $Y_s = y$ and $Y_t  = y$ are events with positive probability.
	
	As argued in~\cite[Sec. 5]{MR2810913}, the sampler $\Gamma B(s)$ may equivalently skip drawing the largest index $I$ and instead generate \[M_i \longleftarrow \Pois(A(s^i)/i)\]
	for all $i \ge 2$, so that the second loop runs over all $i \ge 2$ as well.  The final forest consists of the disjoint union of $i$ copies of $T_{i,j} \longleftarrow \Gamma A(s^i)$ for all $i \ge 2$ and $1 \le j \le M_i$. Its size is given by 
	\[
		\sum_{i \ge 2} \sum_{j=1}^{M_i} i |T_{i,j}|.
	\]
	
	Let $m_i \ge 0$ be a nonnegative integer for all $i \ge 2$ so that $m_i = 0$ for all but finitely many $i \ge 2$. For each $i \ge 2$ and $1 \le j \le m_i$ let $t_{i,j}$ be a P\'olya tree, such that
	\[
		\sum_{i \ge 2} \sum_{j=1}^{m_i} i |t_{i,j}|= y.
	\]
	We have
	\begin{align*}
		&\Prb{ M_i  = m_i \text{ and }  T_{i,j} = t_{i,j} \text{ for all $i \ge 2$, $1 \le j \le m_i$}} \\
		&\quad = \left( \prod_{i \ge 2} \frac{ (A(s^i)/i)^{m_i} / m_i! }{ \exp(A(s^i)/i )} \right) \prod_{i \ge 2} \prod_{1 \le j \le m_i} \frac{ s^{i|t_{i,j}|} }{A(s^i)} \\
		&\quad = \frac{s^y}{\Phi(s)} \prod_{i \ge 2}\frac{1}{i^{m_i} m_i!}.
	\end{align*}
	When restricting to configurations that correspond to a forest with $y$ vertices, only the factor $\prod_{i \ge 2}\frac{1}{i^{m_i} m_i!}$ depends on the actual forest. Therefore, the conditional distribution $(F_s \mid Y_s = y)$ does not depend on the parameter $s$. 
\end{proof}

\section{Randomized tilting of the Boltzmann distribution}
\label{sec:tilt}

We fix $n\ge2$ and set $N = n - 1$. Choose a fixed constant $\epsilon>0$ and define
\begin{align*}
	\lambda = \min\left(\frac{\epsilon}{\sqrt N}, \frac{1}{4} \log(1 / \rho)\right).
\end{align*}
Then
\[
	\rho e^\lambda \le \rho^{3/4} < \sqrt\rho,
\]
so the tilted variables $Y_{\rho e^{\pm\lambda}}$ are well-defined. For triples $(x,y,f)$ with $(x,y) \in \ndN_0^2 \setminus \{(0,0)\}$ and $f$ a forest with $y = |f|$ we set
\begin{align}
	\label{eq:qdist}
	q_{N}(x,y,f) = \frac{\Pr{X_\rho = x} \Pr{F_\rho = f} \cosh(\lambda(x+y))}{B_N}.
\end{align}
with 
\begin{align*}
	B_N = \sum_{(x,y)\ne(0,0)}\Pr{X_\rho = x} \Pr{Y_\rho = y}\cosh(\lambda(x+y)).
\end{align*}
Note that $q_N$ is a probability mass function. The normalising constant $B_N$ is explicit.  For $\sigma \in \{+, -\}$ define
\begin{align*}
	C_\sigma = \exp(e^{\sigma\lambda}-1) \frac{\Phi(\rho e^{\sigma\lambda})}{\Phi(\rho)} - \rho.
\end{align*}
Because $X_\rho$ and $Y_\rho$ are independent,
\[
\Exb{e^{\sigma\lambda(X_\rho + Y_\rho)}} = \exp(e^{\sigma\lambda}-1)
\frac{\Phi(\rho e^{\sigma\lambda})}{\Phi(\rho)}.
\]
Since $\Prb{X_\rho=0, Y_\rho=0} = \rho$, it follows that
\begin{align*}
	C_\sigma = \Exb{e^{\sigma\lambda(X_\rho + Y_\rho)} , (X_\rho,Y_\rho)\ne(0,0)}
\end{align*}
and therefore
\begin{align}
	\label{eq:bn}
	B_N = \frac{C_+ + C_-}{2}.
\end{align}

Consider the following procedure.

\begin{boxedalgorithm}{alg:qn}{Sampler for $q_N$}
			\begin{enumerate}
				\item Choose a sign $\sigma_0 \in \{+, -\}$ with probability
				\begin{align*}
					\Prb{\sigma_0 = +}=\frac{C_+}{C_+ + C_-}, \qquad \Prb{\sigma_0 = -}=\frac{C_-}{C_+ + C_-}.
				\end{align*}
				\item Sample independently
				\begin{align*}
					X &\longleftarrow \Pois(e^{\sigma_0 \lambda}), \\
					F &\longleftarrow \Gamma B(\rho e^{\sigma_0 \lambda}), \\
					Y &:= |F|.
				\end{align*}
				\item If $(X,Y)=(0,0)$, resample $(X,Y,F)$ under the same sign $\sigma_0$ until~$(X,Y) \ne (0,0)$.
				\item Return $(X,Y,F)$.
			\end{enumerate}
\end{boxedalgorithm}

\noindent Here $|F|$ denotes the total number of vertices in the forest $F$.

\begin{lemma}
	The triple  $(X,Y,F)$ returned by the preceding procedure is distributed according to $q_N$.
\end{lemma}
\begin{proof}
	For $\sigma \in \{+,-\}$ we have conditional on $\sigma_0=\sigma$ that the result of the sampling procedure is equal to a given triple $(x,y,f)$ with  $(x,y) \in \ndN_0^2 \setminus \{(0,0)\}$ and $f$ a forest with $|f|=y$ with probability
	\begin{align}
		\label{eq:expression}
		\frac{\Prb{ \Pois(e^{\sigma \lambda}) = x } \Prb{F_{\rho e^{\sigma \lambda}} = f}}{\Prb{ (\Pois(e^{\sigma \lambda}),  Y_{\rho e^{\sigma \lambda}}) \ne (0,0)}}.
	\end{align}
	Since $A(\rho)=1$, we have $X_\rho\sim \Pois(1)$, and therefore
	\begin{align*}
		\Prb{\Pois(e^{\sigma\lambda}) = x} = \Prb{X_\rho=x} \exp\left(\sigma \lambda x + 1 - e^{\sigma \lambda}\right).
	\end{align*}
	Moreover, from
	\[
		\Exb{x^{Y_s}} = \frac{\Phi(sx)}{\Phi(s)}
	\]
it follows that
\begin{align*}
	\Prb{Y_{\rho e^{\sigma\lambda}}=y} = \Prb{Y_\rho=y} e^{\sigma\lambda y} \frac{\Phi(\rho)}{\Phi(\rho e^{\sigma\lambda})}.
\end{align*}
By Proposition~\ref{pro:invariant} this implies
\begin{align*}
	\Prb{F_{\rho e^{\sigma\lambda}} = f} = \Prb{F_\rho = f} e^{\sigma \lambda y} \frac{\Phi(\rho)}{\Phi(\rho e^{\sigma \lambda})}.
\end{align*}
Consequently, the numerator in~\eqref{eq:expression} simplifies to
\begin{align*}
	\Prb{\Pois(e^{\sigma\lambda}) = x} \Prb{F_{\rho e^{\sigma\lambda}} = f} 
		&= \frac{\Prb{X_\rho = x} \Prb{F_\rho = f} e^{\sigma\lambda(x+y)}}{
		\exp(e^{\sigma\lambda} - 1) \Phi(\rho e^{\sigma\lambda})/\Phi(\rho)} \\
		&= \frac{\Prb{X_\rho = x} \Prb{F_\rho = f} e^{\sigma\lambda(x+y)}}{\rho + C_\sigma}.
\end{align*}
Likewise it follows that the denominator in~\eqref{eq:expression} is given by
\begin{align*}
	\Prb{ (\Pois(e^{\sigma \lambda}),  Y_{\rho e^{\sigma \lambda}}) \ne (0,0)} &= 1 - \frac{\Prb{X_\rho = 0} \Prb{Y_\rho = 0} }{\rho + C_\sigma} \\
	&= \frac{C_\sigma}{\rho + C_\sigma}.
\end{align*}
Hence~\eqref{eq:expression} reduces to
\begin{align*}
			\frac{\Prb{ \Pois(e^{\sigma \lambda}) = x } \Prb{F_{\rho e^{\sigma \lambda}} = f}}{\Prb{ (\Pois(e^{\sigma \lambda}),  Y_{\rho e^{\sigma \lambda}}) \ne (0,0)}} = \frac{\Prb{X_\rho = x} \Prb{F_\rho = f} e^{\sigma\lambda(x+y)}}{C_\sigma}.
\end{align*}
Using~\eqref{eq:bn} it follows that
\begin{align*}
	\Prb{(X,Y,F) = (x,y,f)} &= \sum_{\sigma \in \{+,-\}} \frac{\Prb{X_\rho = x} \Prb{F_\rho = f} e^{\sigma\lambda(x+y)}}{C_+ + C_-} \\
	&= \frac{\Pr{X_\rho = x} \Pr{F_\rho = f} \cosh(\lambda(x+y))}{B_N}.
\end{align*}
This completes the proof.
\end{proof}

\section{Sampler for the bridge}

\label{sec:bridge}

For $1 \le k \le N$, let
\begin{align*}
	c_k = \left\lfloor\frac Nk\right\rfloor
\end{align*}
and
\[
	r_k = N - k c_k \in \{0, \ldots, k-1\}.
\]
Define
\begin{align*}
	D_N(k) = \cosh(\lambda c_k)^{-(k-r_k)} \cosh(\lambda(c_k+1))^{-r_k}.
\end{align*}
Finally let
\[
	s(k,m) = B_N^k \rho^m\frac{(k+m-1)!}{k! m!}D_N(k)
\]
and 
\begin{align*}
	\widehat{M}_N = \max_{\substack{1 \le k \le N\\ 0\le m \le N}} s(k,m).
\end{align*}
For $1 \le k \le N$ fixed, we have for all $0 \le m < N$ that
\begin{align*}
	\frac{s(k,m+1)}{s(k,m)} = \rho \frac{k+m}{m+1}.
\end{align*}
This ratio decreases monotonically in $m$. Consequently, $m \mapsto s(k,m)$ attains its maximum at $0$ or $N$ or in the  set $\ndZ \cap [(\rho k-1)/(1-\rho), (\rho k-1)/(1-\rho)+1]$. Hence $\widehat{M}_N$ may be found by checking $O(1)$ values of $m$ for each $k$, hence in $O(N)$ arithmetic operations once $B_N$ is known.

\begin{lemma}
	\label{le:bern}
	For every vector $(u_1,\ldots,u_k) \in \ndN^k$ with $\sum_{i=1}^k u_i=N$,
	\begin{align*}
		\prod_{i=1}^k\cosh(\lambda u_i)^{-1}\le D_N(k).
	\end{align*}
\end{lemma}

\begin{proof}
	The function $g(t) = \log\cosh(\lambda t)$ is convex on $[0,\infty)$.  If two positive integers $a,b$ satisfy $b\ge a+2$, then convexity gives
	\[
		g(a) + g(b) \ge g(a+1) + g(b-1).
	\]
	Repeatedly applying this balancing operation shows that, among all  $(u_1,\ldots,u_k) \in \ndN^k$ with $\sum_{i=1}^k u_i=N$, the sum $\sum_{i=1}^k g(u_i)$ is minimised when the entries are as equal as possible: $k-r_k$ entries equal to $c_k$ and $r_k$ entries equal to $c_k + 1$. 
	Therefore
	\[
	\sum_{i=1}^k\log\cosh(\lambda u_i) \ge (k-r_k)\log\cosh(\lambda c_k) + r_k \log\cosh(\lambda(c_k+1)).
	\]
	Exponentiating the negative of this inequality gives
	\[
			\prod_{i=1}^k\cosh(\lambda u_i)^{-1} \le \cosh(\lambda c_k)^{-(k-r_k)} \cosh(\lambda(c_k+1))^{-r_k}.
	\]
	This completes the proof.
\end{proof}

Consider the following sampler.

\begin{boxedalgorithm}{alg:bridge}{Sampler for the bridge}
Repeat the following attempt until it is accepted.
\begin{enumerate}
	\item Sample independent triples
	\[
		(X_i,Y_i, F_i), \qquad i=1,2,\ldots,
	\]
	from $q_N$ using Algorithm~\ref{alg:qn} until
	\[
		S_j = \sum_{i=1}^j(X_i+Y_i)
	\]
	satisfies $S_j\ge N$.
	\item If $S_j>N$, reject this attempt and restart. If $S_j = N$, put $K = j$ and
	\[
		M = 1 + \sum_{i=1}^K(X_i - 1).
	\]
	If $M < 0$, reject this attempt and restart.
	
	\item Make a Bernoulli test with probability
	\begin{align*}
		R = \frac{ B_N^K \rho^M \frac{(K + M - 1)!}{K! M!} \prod_{i=1}^K \cosh(\lambda(X_i+Y_i))^{-1}} {\widehat{M}_N}.
	\end{align*}
	If this Bernoulli test fails, reject this attempt and restart.
	
	\item Let $L = K + M$.  Choose a uniformly random $K$-element subset of $\{1,\ldots,L\}$.  Insert the list entries of $(X_1, Y_1, F_1), \ldots, (X_K, Y_K, F_K)$ in those positions, preserving their order, and fill the other positions with $(0, 0, \emptyset)$.  This produces a list $(\xi_1, \zeta_1, \tau_1), \ldots, (\xi_L, \zeta_L, \tau_L)$ of length $L$.
	
	\item Return $(\xi_1, \zeta_1, \tau_1), \ldots, (\xi_L, \zeta_L, \tau_L)$.
\end{enumerate}
\end{boxedalgorithm}

Note that Lemma~\ref{le:bern} ensures that in the sampler  we have $R \in [0,1]$, hence the step with the Bernoulli test is valid. There is no overlap of notation with $X_i$ and $Y_i$ for $i \in \ndN$ with the random variates $X_s$ for $0 < s \le \rho < 1$ and $Y_t$ for $0 < t < \sqrt{\rho} < 1$.

The output \[
	(\xi_1, \zeta_1, \tau_1), \ldots, (\xi_L, \zeta_L, \tau_L)
\]
of Algorithm~\ref{alg:bridge} satisfies 
\begin{align}
	\sum_{i=1}^L (\zeta_i + 1) = n
\end{align}
and
\begin{align}
	\sum_{i=1}^L (\xi_i -1) = -1. 
\end{align}
Moreover, it holds that $\zeta_i = |\tau_i|$ for $1 \le i \le L$. In this sense, we may call the output of Algorithm~\ref{alg:bridge} an admissible bridge:

\begin{definition}
	We use the term \emph{admissible bridge} of length $1 \le \ell \le n$ to denote a family $(x_i, y_i,f_i)_{1 \le i \le \ell}$ with  $x_i, y_i \in \ndN_0$ satisfying $\sum_{i=1}^\ell (x_i -1) = -1$ and $\sum_{i=1}^\ell (y_i + 1) = n$  and $f_i$ a forest of size $y_i$ for all $1 \le i \le \ell$. We say an admissible bridge is an \emph{admissible path} if additionally $\sum_{i=1}^r (x_i - 1) \ge 0$ for all $1 \le r < \ell$.
\end{definition}

Recall that $F_\rho$ denotes the random forest produced by $\Gamma B(\rho)$.

\begin{lemma}
	\label{le:main}
	Let $(x_i, y_i,f_i)_{1 \le i \le \ell}$ denote an admissible bridge. During one attempt in Algorithm~\ref{alg:bridge}, the probability of accepting and producing the output list
	\[
		(\xi_i, \zeta_i, \tau_i)_{1 \le i \le L} = (x_i, y_i,f_i)_{1 \le i \le \ell}
	\]
	is given by
	\begin{align*}
		\frac{1}{\widehat{M}_N}\frac{1}{\ell} \prod_{i=1}^\ell  \Prb{X_\rho = x_i} \Prb{F_\rho = f_i}.
	\end{align*}
\end{lemma}

\begin{proof}
	Write
	\[
	I = \{i \in \{1, \ldots, \ell\} \mid (x_i,y_i) \ne (0,0)\} = \{i_1 < \cdots < i_k\}
	\]
	for the set of nonzero positions, and put $m = \ell - k$.  For
	$1 \le r \le k$, set
	\[
		(\bar{x}_r, \bar{y}_r, \bar{f}_r) = (x_{i_r}, y_{i_r}, f_{i_r}).
	\]
	Then 
	\[
		\sum_{r=1}^k (\bar{x}_r + \bar{y}_r) = \sum_{i=1}^\ell (x_i + y_i) = (\ell - 1) + (n -\ell) =  N.
	\]
	Hence, if during one attempt in Algorithm~\ref{alg:bridge} we accept and produce an output list $(x_i, y_i, f_i)_{1 \le i \le \ell}$, then in the sampler we had $K=k$, $L = \ell$, $M = \ell - k$, and
	\[
		(X_r, Y_r, F_r) = (\bar{x}_r, \bar{y}_r, \bar{f}_r), \qquad 1 \le r \le k,
	\]
	and $I$ equal to the uniformly at random chosen subset of $[\ell]$ in the fourth step. The probability of sampling these first two coordinates in this order is
	\[
		\prod_{r=1}^k q_N(\bar{x}_r, \bar{y}_r, \bar{f}_r).
	\]
	Conditional on this event, the probability of passing the Bernoulli test in the third step is given by
	\[
		\frac{1}{\widehat{M}_N} B_N^k \rho^m \frac{(k+m-1)!}{k! m!} \prod_{r=1}^k \cosh(\lambda(\bar{x}_r + \bar{y}_r))^{-1}.
	\]
	Finally, the probability of choosing the set $I$ as insertion positions in the fourth step is given by
	\[
		\binom{\ell}{k}^{-1}.
	\]
	Multiplying the three contributions gives by~\eqref{eq:qdist}, $\Prb{X_\rho=0}\Prb{F_\rho=\emptyset}=\rho$ and $k+m=\ell$ that
	\begin{align*}
		&\Prb{\text{this attempt accepts and outputs the target list}}  \\
		&\quad = \left(\prod_{r=1}^k q_N(\bar{x}_r,\bar{y}_r, \bar{f}_r)\right) \frac{1}{\widehat{M}_N} B_N^k \rho^m \frac{(k+m-1)!}{k! m!} \left(\prod_{r=1}^k \cosh(\lambda(\bar{x}_r + \bar{y}_r))^{-1}\right)  \frac{k!\,m!}{\ell!} \\
		&\quad = \left( \prod_{r=1}^k \Pr{X_\rho = \bar{x}_r} \Pr{F_\rho = \bar{f}_r}\right) \frac{\rho^m}{\widehat{M}_N} \frac{1}{\ell} \\
		&\quad = \frac{1}{\widehat{M}_N}\frac{1}{\ell} \prod_{i=1}^\ell  \Prb{X_\rho = x_i} \Prb{F_\rho = f_i}.
	\end{align*}
	This completes the proof.
\end{proof}

We recall the following standard result known as the cycle lemma~\cite{zbMATH04102170,zbMATH03476417,zbMATH01195781}.

\begin{proposition}[{\cite[Lem. 15.3]{MR2908619}}]
	\label{pro:cycle}
	For any sequence $k_1, \ldots, k_\ell$ of non-negative integers with $\sum_{i=1}^\ell (k_i -1) = -1$ there exists exactly one integer $1 \le t \le \ell$ such that the shifted sequence
	\[
	(\bar{k}_1, \ldots, \bar{k}_\ell) = (k_{t+1}, \ldots, k_\ell, k_1, \ldots, k_{t})
	\]
	satisfies for all $1 \le r < \ell$
	\[
	\sum_{i=1}^r (\bar{k}_i - 1) \ge 0.
	\]
	The integer $t$ is given by the smallest index $1 \le t \le \ell$ such that the sequence
	\[
	h_j = \sum_{i=1}^j (k_i -1), \qquad 1 \le j \le \ell
	\]
	attains its minimum at $j=t$.
\end{proposition}

\begin{lemma}
	\label{le:evac}
	Let $p_N$ denote the probability that a single attempt in Algorithm~\ref{alg:bridge} is accepted. Then
	\[
		p_N = \frac{a_n \rho^n}{\widehat{M}_N}.
	\]
	Consequently Algorithm~\ref{alg:bridge} terminates almost surely. Moreover, for every 
	$(x_i,y_i,f_i)_{1\le i\le \ell}$ satisfying the conditions in Lemma~\ref{le:main}, the probability that Algorithm~\ref{alg:bridge}, including all repeated attempts until the first acceptance, returns this list is given by
	\[
	\frac{1}{a_n \rho^n}\frac{1}{\ell} \prod_{i=1}^{\ell} \Prb{X_\rho=x_i}\Prb{F_\rho=f_i}.
	\]
\end{lemma}
\begin{proof}
	By Lemma~\ref{le:main} and Proposition~\ref{pro:cycle} it follows that
	\begin{align*}
		p_N &= \frac{1}{\widehat{M}_N} \sum_{\ell=1}^{n} \frac{1}{\ell} \sum_{\substack{(x_i,y_i,f_i)_{1\le i\le \ell} \\ \text{admissible bridge}}} \prod_{i=1}^{\ell} \Prb{X_\rho=x_i}\Prb{F_\rho=f_i} \\
		&= \frac{1}{\widehat{M}_N} \sum_{\ell=1}^{n} \sum_{\substack{(x_i,y_i,f_i)_{1\le i\le \ell} \\ \text{admissible path}}} \prod_{i=1}^{\ell} \Prb{X_\rho=x_i}\Prb{F_\rho=f_i}.
	\end{align*}
	
	The last sum (without the $1 / \widehat{M}_N$ prefactor) is precisely the probability that the recursive Boltzmann
	sampler $\Gamma A(\rho)$ produces a tree of size $n$. Since $\Gamma A(\rho)$ has the Boltzmann distribution and $A(\rho)=1$, this probability is $a_n\rho^n$. Therefore
	\[
		p_N = \frac{a_n\rho^n}{\widehat{M}_N}.
	\]

	It remains to account for the repetition of attempts. If $E_b$ denotes
	the event that a single attempt in Algorithm~\ref{alg:bridge} accepts and outputs a fixed admissible bridge list $b = (x_i,y_i,f_i)_{1\le i\le \ell}$, then the probability that Algorithm~\ref{alg:bridge} terminates eventually with $b$ is
	\[
		\sum_{j\ge 0}(1-p_N)^{j}\Prb{E_b} = \frac{\Prb{E_b}}{p_N}.
	\]
	By Lemma~\ref{le:main} it follows that this probability is equal to
	\[
		\frac{1}{a_n \rho^n} \frac{1}{\ell} \prod_{i=1}^{\ell} \Prb{X_\rho=x_i}\Prb{F_\rho=f_i}.
	\]
\end{proof}

\section{Sampler for P\'olya trees}

\label{sec:poly}

We present the following sampler for P\'olya trees.

\begin{boxedalgorithm}{alg:polya}{Sampler for P\'olya trees}
	\begin{enumerate}
		\item Sample the list
		\[
			(\xi_1, \zeta_1, \tau_1), \ldots, (\xi_L, \zeta_L, \tau_L)
		\]
		using Algorithm~\ref{alg:bridge}.
		\item Let
		\[	
			H_j = \sum_{i=1}^j(\xi_i - 1), \qquad 1 \le j \le L.
		\]
		Let $s \ge 1$ be the first index at which $H_j$ is minimal. Produce the cyclically rotated list of triples whose first-coordinate sequence is
		\[
			(\xi_{s+1}, \ldots, \xi_L, \xi_1, \ldots, \xi_s).
		\]
		\item  Interpret this rotated sequence as the depth-first-search ordered list of outdegrees of vertices $v_{s+1}, \ldots, v_L, v_1, \ldots, v_s$ of a plane tree $\tau$.
		\item For each $1 \le i \le L$ attach the forest $\tau_i$ to the vertex $v_i$ of $\tau$ by adding an edge between $v_i$ and each of the root vertices of the trees in $\tau_i$. Forget the plane ordering and the temporary vertex labels.
		\item Return the resulting tree.
	\end{enumerate}
\end{boxedalgorithm}

Our assumption $n \ge 2$ is not a restriction, since there exists exactly one P\'olya tree with one vertex, allowing us to handle this case in a trivial manner.

\begin{theorem}
	The output of the preceding algorithm is uniformly distributed among all P\'olya trees with $n$ vertices.
\end{theorem}
\begin{proof}
	Let $b = (x_i, y_i,f_i)_{1 \le i \le \ell}$ denote an admissible path of length $1 \le \ell \le n$.
	
	It follows that the event $\mathfrak{E}$ for the cyclic shift of $(\xi_i, \zeta_i, \tau_i)_{1 \le i \le L}$ in the second step to equal $b$ corresponds to precisely $\ell$ equally probable outputs of invoking Algorithm~\ref{alg:bridge} in the first step.
	
	By Lemma~\ref{le:evac}, it follows that 
	\[
		\Prb{\mathfrak{E}} = \frac{1}{a_n \rho^n} \prod_{i=1}^\ell  \Prb{X_\rho = x_i} \Prb{F_\rho = f_i}.
	\]
	
	On the other hand, the recursive Boltzmann sampler in Algorithm~\ref{alg:gammaa}, with $t=\rho$ and recursive calls listed in depth-first-search order, draws during the $i$th recursive call an independent copy $\overline{X}_i$ of $X_\rho$ and an independent copy $\overline{F}_i$ of $F_\rho$. It terminates as soon as $\sum_{i=1}^j (\overline{X}_i -1) = -1$. The total size of the generated tree is then given by $\sum_{i=1}^j (1 + |\overline{F}_i|)$.
	
	The probability for the event $\overline{\mathfrak{E}}$ that the sampler $\Gamma A(\rho)$ draws precisely the values $(x_i, f_i)_{1 \le i \le \ell}$ in this way is given by
	\[
		\Prb{\overline{\mathfrak{E}}} = \prod_{i=1}^\ell  \Prb{X_\rho = x_i} \Prb{F_\rho = f_i}.
	\]
	Furthermore, the tree generated by $\Gamma A(\rho)$ in the event $\overline{\mathfrak{E}}$ is identical to the tree generated by Algorithm~\ref{alg:polya} in the event~$\mathfrak{E}$. 
	
	Since the factor $(a_n \rho^n)^{-1}$ only depends on $n$, it follows that the outcome of Algorithm~\ref{alg:polya} is distributed like the Boltzmann distribution of $\Gamma A(\rho)$ conditioned on producing a tree with $n$ vertices. In other words, it is uniformly distributed among all P\'olya trees with $n$ vertices. This concludes the proof.
\end{proof}

\section{Expected runtime}

\begin{lemma}
	\label{le:rbound}
	There are constants $C,c>0$ depending only on $\epsilon$ such that for all $n \ge 2$, $1 \le k \le N = n-1$ and $0  \le m \le N$
	\begin{align}
		\label{eq:1bn}
		B_N^k &\le C (1- \rho)^k,\\
		\label{eq:2dn}
		D_N(k) &\le \exp\left(-c\min\left(\frac Nk,\sqrt N\right)\right),\\
		\label{eq:3skm}
		s(k,m) &\le C k^{-3/2}
		\exp\left(-c\min\left(\frac Nk,\sqrt N\right)\right).
	\end{align}
	Consequently,
	\begin{align}
		\label{eq:mn}
		\widehat{M}_N = O(N^{-3/2}).
	\end{align}
\end{lemma}
\begin{proof}
	All constants in this proof may depend on $\epsilon$, but not on $N$, $k$ or $m$. The random variable $Y_\rho$ has probability generating series
	\[
		\Exb{x^{Y_\rho}} = \frac{\Phi(\rho x)}{\Phi(\rho)}.
	\]
	Since $\Phi$ has radius of convergence $\sqrt{\rho} > \rho$ it follows that $Y_\rho$ has finite exponential moments. Consequently, $X_\rho + Y_\rho$ has finite exponential moments.
	
	For all sufficiently large $n$ we have $\lambda = \epsilon / \sqrt{N}$, hence
	\begin{align}
		\label{eq:olambda}
		\lambda = O(1/\sqrt{N}).
	\end{align}
	Since	
	\[
		B_N = \Exb{\cosh(\lambda (X_\rho + Y_\rho)), X_\rho + Y_\rho > 0},
	\]
	and $\Prb{X_\rho + Y_\rho > 0} = 1 - \rho$ and $\lambda \to 0$ as $n \to \infty$ it follows that
	\[
		B_N = 1 - \rho + O(\lambda^2).
	\]
	By~\eqref{eq:olambda} it follows that
	\[
		B_N / (1 - \rho) = 1 + O(1/N) 
	\]
	and hence for all $1 \le k \le N$
	\[
		(B_N / (1 - \rho) )^k = (1 + O(1/N))^k \le \exp(1 + O(k/N))  = O(1).
	\]
	Hence there exists a constant $C_1>0$ that only depends on $\epsilon$ such that
	\[
		B_N^k \le C_1 (1 - \rho)^k
	\]
	for all $1 \le k \le N$. This proves~\eqref{eq:1bn}.
	
	Next, put $g(t) = \log \cosh(\lambda t)$. By convexity, Jensen's inequality, and $N = c_k k + r_k$ we obtain
	\begin{align*}
		-\log D_N(k) 	&= (k-r_k)g(c_k)+r_k g(c_k+1) \\ 
						&\ge kg( (1 - r_k/k)c_k + (r_k/k) (c_k +1) ) \\
						&= k g(N/k).
	\end{align*}
	It is easy to verify (by separating the cases $0 \le x <1$ and $x>1$) that
	\[
		\log \cosh x \ge \frac{1}{4} \min(x^2, x), \qquad x \ge 0.
	\]
	Since $\lambda=\epsilon/\sqrt N$ for all sufficiently large $n$, this yields
	\[
		-\log D_N(k) \ge \frac{k}{4} \min\left(\frac{\epsilon^2 N^2}{k^2N}, \frac{\epsilon N}{k\sqrt N}\right) = \frac{1}{4} \min\left(\frac{\epsilon^2 N}{k}, \epsilon \sqrt{N}\right)
	\]
	for large enough $n$ and $1 \le k \le N$. Since $D_N(k) < 1$  it follows that there exists $c>0$ with
	\[
		D_N(k) \le \exp\left(- c \min\left(\frac{N}{k}, \sqrt{N}\right)\right)
	\]
	for all $n \ge 2$ and $1 \le k \le N$. This proves~\eqref{eq:2dn}.
	
	Next, let $Z_k$ denote a random variable with negative binomial distribution
	\[
		\Prb{Z_k = m} = \binom{k+m-1}{m} (1 - \rho)^k \rho^m, \qquad m \ge 0.
	\]
	Then $Z_k$ is distributed like the sum of $k$ independent copies of a geometric random variable. By the local limit theorem~\cite[Thm.~4.2.1]{MR0322926} it follows that there exists a constant $C_2>0$ such that
	\begin{align*}
		\sup_{m \ge 0} \Prb{Z_k = m} \le C_2 k^{-1/2}, \qquad k\ge1.
	\end{align*}
	Using~\eqref{eq:1bn} and~\eqref{eq:2dn} it follows that for all $n \ge 2$ and all $1 \le k \le N$, $0\le m\le N$,
	\begin{align*}
		s(k,m) 	&= B_N^k \rho^m \frac{(k+m-1)!}{k!m!} D_N(k) \\
		&\le C_1 (1 - \rho)^k \rho^m k^{-1} \binom{k + m - 1}{m} D_N(k) \\
		&= C_1 k^{-1} \Prb{Z_k = m} D_N(k) \\
		&\le C_1 C_2 k^{-3/2} D_N(k) \\
		&\le  C_1 C_2 k^{-3/2}\exp\left(-c\min\left(\frac Nk,\sqrt N\right)\right).
	\end{align*}
	Clearly~\eqref{eq:1bn} remains valid if we replace $C$ by a larger constant. Hence this completes the verification of~\eqref{eq:1bn},~\eqref{eq:2dn} and~\eqref{eq:3skm}.

	Finally, with $x := N/k \le N$ we have
	\begin{align*}
		s(k,m) 	\le C N^{-3/2} x^{3/2} \exp\left(-c\min(x,\sqrt N)\right).
	\end{align*}
	We have for $0 \le x \le \sqrt{N}$
	\[
		x^{3/2} \exp\left(-c\min(x,\sqrt N)\right) = x^{3/2} \exp\left(-c x\right) \le \sup_{t \ge 0} t^{3/2} \exp\left(-c t\right) < \infty,
	\]
	and for $\sqrt{N} \le x \le N$
	\[
		x^{3/2} \exp\left(-c\min(x,\sqrt N)\right) \le N^{3/2} \exp\left(-c \sqrt{N} \right) \le \sup_{t \ge 0} t^{3/2} \exp\left(-c \sqrt{t}\right).
	\]
	Consequently,
	\begin{align*}
		\widehat{M}_N 	= \max_{\substack{1 \le k \le N\\ 0\le m \le N}} s(k,m) = O(N^{-3/2}).
	\end{align*}
	This completes the proof.
\end{proof}

\begin{lemma}
	\label{le:runtimeingredients}
	Uniformly for $n\ge2$, with $N=n-1$, the following assertions hold.
	\begin{enumerate}
		\item The preprocessing needed for Algorithm~\ref{alg:bridge}, including
		the computation of $\widehat{M}_N$, the factorial table needed in the
		Bernoulli probability $R$, and the constants appearing in $R$, costs
		$O(n)$ arithmetic operations.
		
		\item A single call to Algorithm~\ref{alg:qn} has expected runtime $O(1)$.
		
		\item A single attempt in Algorithm~\ref{alg:bridge} has expected runtime
		$O(n)$.
		
		\item The expected number of attempts made by Algorithm~\ref{alg:bridge} before accepting is bounded uniformly.
		
		\item The expected runtime of Algorithm~\ref{alg:bridge} is $O(n)$.
	\end{enumerate}
\end{lemma}

\begin{proof}
	We first discuss preprocessing. The constants $C_+$, $C_-$, and $B_N$ are computed once from their definitions and from~\eqref{eq:bn}. As observed before Lemma~\ref{le:bern}, once $B_N$ is known the maximum
	\[
		\widehat{M}_N = \max_{\substack{1\le k\le N\\0\le m\le N}} s(k,m)
	\]
	can be found by checking only $O(1)$ candidate values of $m$ for each $1\le k\le N$. Hence the computation of $\widehat{M}_N$ costs $O(n)$ arithmetic operations.
	
	In order to make the cost of the Bernoulli test explicit, we evaluate the acceptance probability $R$ in logarithmic form. Precompute
	\[
		\ell_j = \log(j!), \qquad 0 \le j \le 2N,
	\]
	using  the recurrence $\ell_0 = 0$ and $\ell_j = \ell_{j-1} + \log j$. This costs $O(n)$ arithmetic operations. During the computation of $\widehat{M}_N$, each candidate value of $s(k,m)$ may then be evaluated through
	\[
		\log s(k,m) = k\log B_N + m \log \rho + \ell_{k+m-1} - \ell_k - \ell_m + \log D_N(k),
	\]
	where $D_N(k)$ is defined before Lemma~\ref{le:bern}. This is constant time per candidate. Thus the preprocessing, including the values needed later for $R$, is $O(n)$.
	
	We next prove that one call to Algorithm~\ref{alg:qn} has uniformly bounded
	expected cost. Put
	\[
		\lambda_0 = \frac{1}{4} \log(1/\rho).
	\]
	By the definition of $\lambda$ in the randomised tilting construction, $0 < \lambda \le \lambda_0$, and therefore
	\[
		\rho e^{\sigma\lambda} \le \rho e^{\lambda_0} = \rho^{3/4} < \sqrt{\rho}, \qquad \sigma\in\{+,-\}.
	\]
	The Poisson random variable in Algorithm~\ref{alg:qn} has parameter $e^{\sigma\lambda} \in [e^{-\lambda_0}, e^{\lambda_0}]$. Hence its generation cost is uniformly bounded in expectation.

	It remains to bound the expected runtime for a call to $\Gamma B(\rho e^{\sigma\lambda})$. This follows from Lemma~\ref{le:boltzb} and $\rho e^{\sigma\lambda} \le \rho^{3/4}$.
	
	The rejection step inside Algorithm~\ref{alg:qn} also changes the cost only by a bounded factor. Indeed, for both signs $\sigma$,
	\[
		\Prb{(X,Y)\ne(0,0)} \ge \Prb{\Pois(e^{\sigma\lambda})\ge1} \ge 1-\exp(-e^{-\lambda_0}) >0.
	\]
	Thus the number of resampling trials in Algorithm~\ref{alg:qn} is stochastically dominated by a geometric random variable with uniformly bounded expectation. This proves that one sample from $q_N$, whose correctness is established in the lemma following~\eqref{eq:bn}, has expected cost $O(1)$.
	
	We now consider a single attempt of Algorithm~\ref{alg:bridge}. The triples sampled from $q_N$ satisfy $(X_i,Y_i)\ne(0,0)$ by definition of $q_N$ in~\eqref{eq:qdist}, and hence every increment
	$X_i+Y_i$ is a positive integer. Therefore, before the partial sum
	\[
		S_j=\sum_{i=1}^j(X_i+Y_i)
	\]
	first reaches or exceeds $N$, at most $N$ triples can be sampled. Since sampling one triple costs $O(1)$ in expectation, the expected cost of producing the triples in one attempt is $O(n)$.
	
	It remains to account for the Bernoulli test in the third step of Algorithm~\ref{alg:bridge}. During the sampling of the triples we maintain the accumulated value
	\[
		H_K = \sum_{i=1}^K \log\cosh\bigl(\lambda(X_i+Y_i)\bigr),
	\]
	which costs constant time per sampled triple. If $S_K = N$ and $M \ge 0$, then
	\[
		K\le N, \qquad \sum_{i=1}^K X_i\le N, \qquad K + M - 1 = \sum_{i=1}^K X_i \le N.
	\]
	Using the precomputed table $(\ell_j)$, the logarithm of the Bernoulli probability $R$ can therefore be evaluated in constant time as
	\[
		\log R = K \log B_N + M \log \rho + \ell_{K + M - 1} - \ell_K - \ell_M - H_K - \log \widehat{M}_N.
	\]
	The Bernoulli test may then be performed by drawing $U \sim \mathrm{Unif}(0,1)$ and checking whether $\log U\le \log R$. Lemma~\ref{le:bern}, together with the definition of $\widehat{M}_N$, ensures that $R\in[0,1]$, so this test is valid. Thus calculating $R$ does not increase the order of the
	runtime. Even without maintaining $H_K$ during the sampling, recomputing the product in $R$ would cost only $O(K)\le O(N)$, so the attempt would still have linear expected cost.
	
	The remaining operations in one attempt are also linear. If the attempt is not rejected before the insertion step, then $L=K+M\le n$, and choosing the $K$-element subset of $\{1,\ldots,L\}$ and forming the output list costs $O(L) \le O(n)$. Hence one attempt of Algorithm~\ref{alg:bridge}, including the calculation of $R$, has expected cost $O(n)$.
	
	Finally, we bound the number of attempts in Algorithm~\ref{alg:bridge}. By Lemma~\ref{le:evac}, the acceptance probability of one attempt is
	\[
		p_N = \frac{a_n \rho^n}{\widehat{M}_N}.
	\]
	Otter's coefficient asymptotic~\eqref{eq:coef} gives
	\[
		a_n\rho^n \sim c_A n^{-3/2},
	\]
	while Lemma~\ref{le:rbound}, specifically~\eqref{eq:mn}, gives
	\[
		\widehat{M}_N = O(N^{-3/2}).
	\]
	These two estimates imply that there exists a constant $p_*>0$ such that
	\[
		p_N \ge p_*
	\]
	for all $n\ge2$. Therefore the number of attempts in Algorithm~\ref{alg:bridge} is geometrically distributed with uniformly bounded expectation. 
	
	Since the expected number of attempts is bounded and each attempt has expected runtime $O(n)$, the expected runtime of Algorithm~\ref{alg:bridge} is $O(n)$. This completes the proof.
\end{proof}

\begin{theorem}
	\label{te:runtime}
	The expected runtime of Algorithm~\ref{alg:polya} is $O(n)$ in the
	computational model specified above.
\end{theorem}
\begin{proof}
	By Lemma~\ref{le:runtimeingredients} the expected runtime of Algorithm~\ref{alg:bridge} is $O(n)$. It remains to account for the deterministic work performed after the bridge has been sampled. The output of Algorithm~\ref{alg:bridge} satisfies
	\[
		\sum_{i=1}^L(\zeta_i + 1) = n
	\]
	as stated after Algorithm~\ref{alg:bridge}. In particular, $L \le n$, and the total size of the attached forests is $n-L$. Therefore computing the partial sums
	\[
		H_j = \sum_{i=1}^j(\xi_i - 1), \qquad 1 \le j \le L,
	\]
	finding the first index at which they attain their minimum, performing the cyclic rotation from Algorithm~\ref{alg:polya}, constructing the depth-first-search tree, and attaching all forests together cost deterministic time $O(n)$.
	
	Combining the $O(n)$ preprocessing cost, the $O(n)$ expected cost of Algorithm~\ref{alg:bridge}, and the final deterministic $O(n)$ construction cost proves that Algorithm~\ref{alg:polya} has expected runtime $O(n)$.
\end{proof}

\section*{Acknowledgement}
This research was funded in part by the Austrian Science Fund (FWF) 10.55776/PAT6732623 and 10.55776/PAT1382025. For open access purposes, the authors have applied a CC BY public copyright license to any author-accepted manuscript version arising from this submission.

\bibliographystyle{abbrv}
\bibliography{polya}

\begin{thebibliography}{10}

\bibitem{zbMATH06774439}
A.~Bacher, O.~Bodini, and A.~Jacquot.
\newblock Efficient random sampling of binary and unary-binary trees via
  holonomic equations.
\newblock {\em Theor. Comput. Sci.}, 695:42--53, 2017.

\bibitem{zbMATH08207340}
L.~Bartholdi and P.~Diaconis.
\newblock An algorithm for uniform generation of unlabeled ({P{\'o}lya}) trees.
\newblock {\em Forum Math. Sigma}, 14:27, 2026.
\newblock Id/No e76.

\bibitem{zbMATH06244239}
O.~Bodini, D.~Gardy, and A.~Jacquot.
\newblock Asymptotics and random sampling for {BCI} and {BCK} lambda terms.
\newblock {\em Theor. Comput. Sci.}, 502:227--238, 2013.

\bibitem{zbMATH06683520}
O.~Bodini and Y.~Ponty.
\newblock Multi-dimensional {Boltzmann} sampling of languages.
\newblock In {\em Proceeding of the 21st international meeting on
  probabilistic, combinatorial, and asymptotic methods in the analysis of
  algorithms (AofA'10), Vienna, Austria, June 28 -- July 2, 2010}, pages
  49--64. Nancy: The Association. Discrete Mathematics \& Theoretical Computer
  Science (DMTCS), 2010.

\bibitem{zbMATH06109904}
O.~Bodini, O.~Roussel, and M.~Soria.
\newblock Boltzmann samplers for first-order differential specifications.
\newblock {\em Discrete Appl. Math.}, 160(18):2563--2572, 2012.

\bibitem{MR2810913}
M.~Bodirsky, {\'E}.~Fusy, M.~Kang, and S.~Vigerske.
\newblock Boltzmann samplers, {P}\'olya theory, and cycle pointing.
\newblock {\em SIAM J. Comput.}, 40(3):721--769, 2011.

\bibitem{zbMATH05039060}
M.~Bodirsky and M.~Kang.
\newblock Generating outerplanar graphs uniformly at random.
\newblock {\em Comb. Probab. Comput.}, 15(3):333--343, 2006.

\bibitem{zbMATH06195431}
K.~Bringmann and T.~Friedrich.
\newblock Exact and efficient generation of geometric random variates and
  random graphs.
\newblock In {\em Automata, languages, and programming. 40th international
  colloquium, ICALP 2013, Riga, Latvia, July 8--12, 2013, Proceedings, Part I},
  pages 267--278. Berlin: Springer, 2013.

\bibitem{MR2888318}
L.~Devroye.
\newblock Simulating size-constrained {G}alton-{W}atson trees.
\newblock {\em SIAM J. Comput.}, 41(1):1--11, 2012.

\bibitem{MR2483235}
P.~Flajolet and R.~Sedgewick.
\newblock {\em Analytic combinatorics}.
\newblock Cambridge University Press, Cambridge, 2009.

\bibitem{fusypanagiotou2026}
{\'E}.~Fusy and K.~Panagiotou.
\newblock {\em Manuscript in preparation}.

\bibitem{MR0322926}
I.~A. Ibragimov and Y.~V. Linnik.
\newblock {\em Independent and stationary sequences of random variables}.
\newblock Wolters-Noordhoff Publishing, Groningen, 1971.
\newblock With a supplementary chapter by I. A. Ibragimov and V. V. Petrov,
  Translation from the Russian edited by J. F. C. Kingman.

\bibitem{MR2908619}
S.~Janson.
\newblock Simply generated trees, conditioned {G}alton-{W}atson trees, random
  allocations and condensation.
\newblock {\em Probab. Surv.}, 9:103--252, 2012.

\bibitem{MR0025715}
R.~Otter.
\newblock The number of trees.
\newblock {\em Ann. of Math. (2)}, 49:583--599, 1948.

\bibitem{zbMATH07755473}
K.~Panagiotou, L.~Ramzews, and B.~Stufler.
\newblock Exact-size sampling of enriched trees in linear time.
\newblock {\em SIAM J. Comput.}, 52(5):1097--1131, 2023.

\bibitem{MR3773800}
K.~Panagiotou and B.~Stufler.
\newblock Scaling limits of random {P}\'{o}lya trees.
\newblock {\em Probab. Theory Related Fields}, 170(3-4):801--820, 2018.

\bibitem{zbMATH01195781}
J.~Pitman.
\newblock Enumerations of trees and forests related to branching processes and
  random walks.
\newblock In {\em Microsurveys in discrete probability. DIMACS workshop,
  Princeton, NJ, USA, June 2--6, 1997}, pages 163--180. Providence, RI: AMS,
  American Mathematical Society, 1998.

\bibitem{MR1577579}
G.~P{\'o}lya.
\newblock Kombinatorische {A}nzahlbestimmungen f\"ur {G}ruppen, {G}raphen und
  chemische {V}erbindungen.
\newblock {\em Acta Math.}, 68(1):145--254, 1937.

\bibitem{Sportiello2021}
A.~Sportiello.
\newblock Boltzmann sampling of irreducible context-free structures in linear
  time, 2021.

\bibitem{zbMATH08207337}
B.~Stufler.
\newblock Probabilistic enumeration and equivalence of nonisomorphic trees.
\newblock {\em Discrete Math. Theor. Comput. Sci.}, 27(3):14, 2025.
\newblock Id/No 18.

\bibitem{zbMATH04102170}
L.~Tak{\'a}cs.
\newblock Ballots, queues and random graphs.
\newblock {\em J. Appl. Probab.}, 26(1):103--112, 1989.

\bibitem{zbMATH03476417}
J.~G. Wendel.
\newblock Left-continuous random walk and the {Lagrange} expansion.
\newblock {\em Am. Math. Mon.}, 82:494--499, 1975.

\bibitem{fusy2026leapgeneratorscompositionschemes}
Éric Fusy and C.~Pivoteau.
\newblock Leap generators for composition schemes.
\newblock {\em arxiv:2605.06471}, 2026.

\end{thebibliography}

\end{document}